\documentclass[a4paper, 12pt]{amsproc}
\allowdisplaybreaks

\usepackage{fullpage}
\usepackage{amsmath, amsthm, amssymb, mathtools, mathrsfs}
\usepackage{algorithm}
\usepackage{algpseudocode}
\usepackage{caption}
\usepackage{mdframed}
\usepackage{tcolorbox}
\tcbuselibrary{skins,breakable}

\usepackage{bm}
\usepackage{forest}

\usepackage{chessfss}
\usepackage{mathtools}
\usepackage{tikz}
\usetikzlibrary{intersections, calc, arrows.meta}

\usepackage{pgfplots}

\usepackage{graphicx}
\usepackage{here}
\usepackage{time}
\usepackage{bbm}

\usepackage{xcolor}
\usepackage{hyperref}
\usepackage[capitalize,nameinlink,noabbrev,nosort]{cleveref}
\hypersetup{
	colorlinks=true, 
	linkcolor=brown, 
	citecolor=brown, 
	filecolor=brown, 
	urlcolor=brown, 
}
\makeatletter
\@namedef{subjclassname@2020}{%
 \textup{2020} Mathematics Subject Classification}
\makeatother

\newtheorem{theoremcounter}{Theorem Counter}[section]

\theoremstyle{definition}
\newtheorem{definition}[theoremcounter]{Definition}

\newtheorem{example}[theoremcounter]{Example}

\theoremstyle{plain}
\newtheorem{lemma}[theoremcounter]{Lemma}
\newtheorem{proposition}[theoremcounter]{Proposition}
\newtheorem{corollary}[theoremcounter]{Corollary}

\newtheorem{theorem}[theoremcounter]{Theorem}

\newtheorem{open}[theoremcounter]{Open problem}

\numberwithin{equation}{section}

\makeatletter
\renewenvironment{proof}[1][\proofname]{%
  \par\pushQED{\qed}%
  \normalfont\topsep6pt \trivlist
  \item[\hskip\labelsep \textbf{#1.}]\ignorespaces 
}{%
  \popQED\endtrivlist\@endpefalse
}
\makeatother

\newtcolorbox{breakalg}[2][]{%
  breakable,
  enhanced,
  sharp corners,
  colframe=black,
  boxrule=0pt,
  top=5pt,
  bottom=5pt,
  title={\textbf{Algorithm~#2}},
  fonttitle=\bfseries,
  colback=white,
  overlay unbroken={
    \draw[black] (frame.north west) -- (frame.north east);
    \draw[black] (title.south west) -- (title.south east);
    \draw[black] (frame.south west) -- (frame.south east);
  },
  overlay first={
    \draw[black] (frame.north west) -- (frame.north east);
    \draw[black] (title.south west) -- (title.south east);
  },
  overlay middle={
  },
  overlay last={
    \draw[black] (frame.south west) -- (frame.south east);
  },
  #1
}

\makeatletter
\tagsleft@false
\makeatother

\newcommand{\N}{\mathbb{N}}
\newcommand{\Z}{\mathbb{Z}}

\newcommand{\C}{\mathbb{C}}
\newcommand{\M}{\mathbb{M}}

\newcommand{\Pri}{\mathbb{P}}
\newcommand{\mA}{\mathcal{A}}
\newcommand{\mB}{\mathcal{B}}
\newcommand{\mC}{\mathcal{C}}
\newcommand{\mS}{\mathcal{S}}
\newcommand{\mM}{\mathcal{M}}
\newcommand{\mE}{\mathcal{E}}
\newcommand{\mQ}{\mathcal{Q}}

\newcommand{\bo}{\mathbf{1}}

\newcommand{\SL}{\operatorname{SL}}
\newcommand{\rk}{\operatorname{rk}}

\newcommand{\APT}{\operatorname{APT}}

\newcommand{\Prim}{\operatorname{Prim}}

\begin{document}

\title[]{On the finiteness of prime trees and their relation to modular forms} 

\author{Yusuke Fujiyoshi} 
\address{Joint Graduate School of Mathematics for Innovation, Kyushu University,
Motooka 744, Nishi-ku, Fukuoka 819-0395, Japan}
\email{fujiyoshi.yusuke.671@s.kyushu-u.ac.jp}


\maketitle

\begin{abstract}
In this paper, we introduce the prime trees associated with a finite subset $P$ of the set of all prime numbers, and provide conditions under which the tree is of finite type. Moreover, we compute the density of finite-type subsets $P$. As an application, we show that for weight $k \ge 2$ and levels $N = N'\prod_{p \in P} p^{a_p}$, where $N'$ is squarefree and $a_{p} \geq 2$, every cusp form $f \in \mS_k(\Gamma_0(N))$ can be expressed as a linear combination of products of two specific Eisenstein series whenever $P$ is of finite type.
\end{abstract}

\section{Introduction}
The space $\mM_{k}(\Gamma_{0}(N))$ of modular forms of weight $k$ and level $N$ decomposes as the direct sum of the Eisenstein subspace $\mE_{k}(\Gamma_{0}(N))$ and the cusp form subspace $\mS_{k}(\Gamma_{0}(N))$.\\

For $N = 1$, it is well known that every cusp form $f \in \mS_{k}(\SL_{2}(\Z))$ is a linear combination of products of the Eisenstein series $E_{4}$ and $E_{6}$ of weights $4$ and $6$, respectively. However, the number of such monomials in $E_4$ and $E_6$ required to express these linear combinations grows linearly with the weight $k$. Zagier \cite{Zagier} showed that, by using the Rankin–Selberg method, for $k \ge 8$, the space $\mM_k(\SL_2(\Z))$ is generated by the set of products $E_{\ell}E_{k-\ell}$ for $4 \le \ell \le k - 4$. Similar results are known to hold for $\mM_{k}(\Gamma_{0}(p))$
for prime $p$ and $k \ge 4$.
More precisely, Imamo\u{g}lu and Kohnen \cite{IK} treated the case $p = 2$,
while Kohnen and Martin \cite{KY} established analogous results for odd primes. Furthermore, Dickson and Neururer \cite{DN} proved that for weights $k \ge 4$ and levels of the form $N = p^{a}q^{b}N'$, where $p$ and $q$ are primes and $N'$ is squarefree, the space $\mM_k(\Gamma_0(N))$ is the sum of $\mE_k(\Gamma_0(N))$ and $\mQ_k(N)$, where $\mQ_{k}(N)$ is generated by the set of products of two
Eisenstein series with Dirichlet characters $E_{\ell}^{\phi,\psi}$.

\begin{theorem}[\cite{DN}]\label{thm1.1}
Let $k \ge 4$ be even, and let $N = p^{a}q^{b}N'$, where $p$ and $q$ are primes, $a,b \in \Z_{\ge 0}$, and $N'$ is squarefree. Then the restriction of the cuspidal projection to $\mQ_k(N)$ is surjective; that is,
\begin{align}\label{(1.1)}
    \mM_k(\Gamma_0(N)) = \mQ_k(N) + \mE_k(\Gamma_0(N)).
\end{align}
\end{theorem}

The case of weight $2$ is different. We define the space $\mS_{2,\rk=0}(\Gamma_0(N))$ to be generated by newforms $f$ and their lifts with nonvanishing central $L$-value $L(f,1)$. Dickson and Neururer then proved an analogue of \cref{thm1.1}.

\begin{theorem}[\cite{DN}]\label{thm1.2}
Let $N$ and $\mQ_2(N)$ be as in \cref{thm1.1}. Then
\begin{align}\label{(1.2)}
    \mS_{2,\rk=0}(\Gamma_0(N)) \oplus \mE_2(\Gamma_0(N))
    = \mQ_2(N) + \mE_2(\Gamma_0(N)).
\end{align}
\end{theorem}

In this paper, we generalize the results of Dickson and Neururer. To this end, we introduce the prime trees associated with a finite subset $P$ of the set of all prime numbers. We define $N(P)$ as the product of all elements of $P$, and define the prime tree associated with $P$ for each $(x,y) \in \N^2$ as follows. Throughout this paper, we denote by $\Pri$ the set of all prime numbers.

\begin{definition}\label{def1.3}
    Let $P$ be a finite subset of $\Pri$. For $(x,y) \in \N^{2}$, we define the \emph{Additive Prime Tree} (APT) as the smallest subset $\APT_{P}(x,y) \subset \N^{2}$ satisfying:

    \begin{itemize}
      \item[(i)] $(x,y) \in \APT_P(x,y)$.
    
      \item[(ii)] If $(X,Y) \in \APT_P(x,y)$ and the following hold:
      \[
        \gcd(X, Y, N(P)) = 1,\ ^{\exists} p, p' \in P\ s.t.\ p\mid X \text{ and } p'\mid Y,
      \]
      then $(X, X + Y),\ (X + Y, Y) \in \APT_P(x,y)$.
    \end{itemize}
    We define a relation $\to$ on the set $\APT_{P}(x,y)$ by
    \[
        (X, Y) \to (Z, W) \overset{\mathrm{def}}{\Longleftrightarrow} (Z,W) = (X, X+Y)\ \text{or}\ (X+Y, Y).
    \]
    for $(X, Y), (Z, W) \in \APT_{P}(x,y)$.
    We say that a set $P$ is of \emph{finite type} if the set $\APT_P(x,y)$ is finite for every $(x,y) \in \N^2$; otherwise, we say that $P$ is of \emph{infinite type}.
\end{definition}

\begin{example}\label{ex1.4}
    Let $P = \{ 2,3,5 \}$ and $(x,y) = (2,3)$.
    
    We illustrate below the structure of $\APT_{P}(2,3)$ for this choice of $P$. 
    Each node represents a pair $(X,Y) \in \APT_{P}(2,3)$, and an arrow $(X,Y) \to (Z,W)$ indicates that $(Z,W)$ is obtained from $(X,Y)$ by one of the two branching rules in \cref{def1.3}.

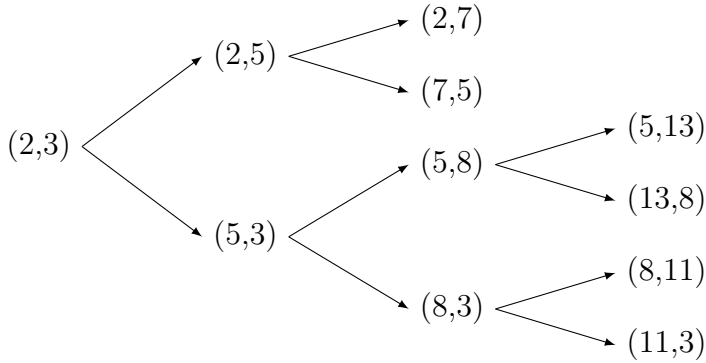
    \begin{forest}
    for tree={
      grow=east,
      parent anchor=east,
      child anchor=west,
      edge={-latex},
      l sep=45pt,
      s sep=5pt,
      anchor=center,
      align=center
    }
    [{(2,3)}
      [{(5,3)}
        [{(8,3)}
          [{(11,3)}]
          [{(8,11)}]
        ]
        [{(5,8)}
          [{(13,8)}]
          [{(5,13)}]
        ]
      ]
      [{(2,5)}
        [{(7,5)}]
        [{(2,7)}]
      ]
    ]
    \end{forest}

    Then
    \begin{align*}
        \APT_{P}(2,3) = \Big\{ (2,3), (2,5),& (5,3), (2,7), (7,5), (5,8),\\ &(8,3), (5,13), (13,8), (8,11), (11,3)\Big\}.
    \end{align*}
\end{example}

We now summarize the fundamental properties of the additive prime trees introduced above.
The following theorem establishes their finiteness and the corresponding density estimates.

\begin{theorem} \label{main}
    Let $P$ be a finite subset of $\Pri$. Then the following statements hold.
    \begin{itemize}
        \item[(1)] $P$ is of finite type for $|P| \leq 4$.
        \item[(2)] Let $|P| = r \geq 5$. Then, there exist only finitely many primitive sets $P$, where the notion of primitive is defined in \cref{def2.5}.
        \item[(3)] We define $P(M) = \{ p:prime \mid p^{2}|M \}$ and
        \[
            \mA(N) = \{ M \leq N \mid P(M) \text{ is of finite type} \}.
        \]
        Then
        \[
            0.99999991254 \leq \lim_{N \to \infty} \frac{|\mA(N)|}{N} \leq 0.9999999153.
        \]
    \end{itemize}
\end{theorem}

The above theorem establishes the finiteness properties of additive prime trees and the asymptotic density of finite-type sets.

In the next result, we relate this finiteness condition to the structure of modular forms.  
When the set $P$ is of finite type, the decomposition of $\mM_k(\Gamma_0(N))$ into $\mE_{k}(\Gamma_{0}(N))$ and $\mQ_{k}(N)$ extends naturally to levels associated with $P$.

\begin{theorem}\label{thm1.6}
    Let $k \geq 2$ be even and let a finite subset $P \subset \Pri$ be of finite type. Then (\ref{(1.1)}) and (\ref{(1.2)}) hold for $N = N'\prod_{p \in P} p^{a_{p}}$ where $a_{p} \in \Z_{\geq 0}$ for all $p \in P$ and $N'$ is squarefree.
\end{theorem}

As an immediate consequence of \Cref{main,thm1.6},
we obtain explicit families of levels for which the decomposition formulas (\ref{(1.1)}) and (\ref{(1.2)}) hold unconditionally, together with an asymptotic estimate for their natural density.

\begin{corollary}\label{cor1.7}
    The following statements hold.
    \begin{enumerate}
        \item[(1)] Let $N = p^{a}q^{b}r^{c}s^{d}N'$,
        where $p,q,r,s$ are primes, $a,b,c,d \in \Z_{\ge0}$,
        and $N'$ is squarefree. Then (\ref{(1.1)}) and (\ref{(1.2)}) hold.
    
        \item[(2)] Define
        \begin{align*}
            \mB(N) := \{\, M \le N \mid\ &
            \text{(\ref{(1.1)}) holds for all even integers $k \ge 4$,}\\
            &\text{and (\ref{(1.2)}) holds for $k = 2$}\,\}.
        \end{align*}
        Then
        \[
            \lim_{N \to \infty} \frac{|\mB(N)|}{N} \ge 0.99999991254.
        \]
    \end{enumerate}
\end{corollary}
\vspace{5mm}
\section{Preliminaries}
Let $P$ be a finite subset of $\Pri$.  
From \cref{main}, we know that most such sets $P$ are of finite type.  
In this section, we provide examples of finite subsets $P$ that are of infinite type, give a criterion for detecting when a set is of infinite type, and describe an algorithm for finding such sets. For $a \in \N$, we define
\[
    N_P(a) := \gcd(a, N(P)).
\]

\begin{lemma}\label{lem2.1}
Let $P$ be a finite subset of $\Pri$, and let $(x_0,y_0)\to\cdots\to(x_i,y_i)$ be a path in $\APT_P(x,y)$.
Set $L=\bigl(\begin{smallmatrix}1&0\\[1pt]1&1\end{smallmatrix}\bigr)$ and $R=\bigl(\begin{smallmatrix}1&1\\[1pt]0&1\end{smallmatrix}\bigr)$.
Then there exists a unique matrix $M_i\in\langle L,R\rangle\subset \mathrm{SL}_2(\Z_{\ge0})$ such that
\[
  (x_i, y_i)=(x_0,y_0)M_i = (A_{i}x_{0}+B_{i}y_{0},C_{i}x_{0}+D_{i}y_{0})
\]
where $M_i=\bigl(\begin{smallmatrix}A_i&C_i\\ B_i&D_i\end{smallmatrix}\bigr)$.
\end{lemma}

\begin{proof}
This follows immediately by induction on the path length.
\end{proof}

\begin{proposition}\label{prop2.2}
Let $P$ be a finite subset of $\Pri$.  
Suppose that there exist $x,y \in \N$ and a sequence $((x_i, y_i))_{0 \le i \le \ell} \subset \APT_P(x,y)$ with $(x_0,y_0) \to \cdots \to (x_\ell,y_\ell)$ and $\ell \ge 1$ such that
\[
    (x_\ell, y_\ell) \in \{ (N_1 s, N_2 (K s + M t)) \mid s,t \in \Z \},
\]
and that
\[
    N_P(x_0) = N_P(x_\ell), \quad N_P(y_0) = N_P(y_\ell),
\]
where
\[
    N_1 = N_P(x_0), \quad
    N_2 = N_P(y_0), \quad
    M = N(P) / (N_1 N_2), \quad
    K = \biggl(\frac{x_0}{N_1}\biggr)^{-1} \biggl(\frac{y_0}{N_2}\biggr) \in (\Z/M\Z)^\times.
\]
Then $\APT_P(x,y)$ is infinite; that is, $P$ is of infinite type.
\end{proposition}

\begin{proof}
Let $N:=N(P)$ and write $M_i=\bigl(\begin{smallmatrix}A_i&C_i\\ B_i&D_i\end{smallmatrix}\bigr)$ as in \cref{lem2.1}. For $0\le i\le \ell$ set
\[
  (X_i, Y_i) := (x_\ell, y_\ell)M_i = 
  (x_\ell, y_\ell)
  \begin{pmatrix}
    A_i & C_i \\[2pt]
    B_i & D_i
  \end{pmatrix}
  = (A_i x_\ell+B_i y_\ell, C_i x_\ell+D_i y_\ell)
\]
We claim that for all $i$,
\[
  N_P(x_i)=N_P(X_i)\quad\text{and}\quad N_P(y_i)=N_P(Y_i).
\]
Since $(X_0,Y_0)=(x_\ell,y_\ell)\in\APT_P(x,y)$, it suffices to prove the equalities of $N_P(\cdot)$.
We show the first; the second is identical with $x_{i}\leftrightarrow y_{i}$.

By the assumption of the proposition, we have
\[
  (x_\ell,y_\ell) = (N_1 s,\; N_2(Ks+Mt))
\]
for some $s,t\in\Z$, where the parameters $N_1,N_2,M,K$ are as defined above.

\smallskip
\emph{(A) Let $p$ be a prime divisor of $N_P(x_i)$. We prove that $p\mid X_i$.}

\smallskip
\noindent
If $p\mid M$, then since $x_i = A_i x_0 + B_i y_0$ and $p\nmid x_0,y_0$, we have
\[
    0 \equiv A_i x_0 + B_i y_0
   = A_i\!\left(\frac{x_0}{N_1}\right)\!N_1 + B_i\!\left(\frac{y_0}{N_2}\right)\!N_2
   \equiv (A_i N_1 + B_i N_2 K)\!\left(\frac{x_0}{N_1}\right) \pmod p,
\]
so $A_i N_1 + B_i N_2 K \equiv 0 \pmod p$.
Therefore,
\[
    X_i = A_i N_1 s + B_i N_2 (Ks+Mt)
   \equiv (A_i N_1 + B_i N_2 K)s \equiv 0 \pmod p.
\]

\smallskip
\noindent
If $p\nmid M$, then $p$ divides either $N_1$ or $N_2$.
Suppose $p\mid N_1$ (the other case is analogous).  
Since $p\mid x_0$ and $p\nmid y_0$, we have
\[
    0 \equiv x_i = A_i x_0 + B_i y_0 \equiv B_i y_0 \pmod p,
\]
so $p\mid B_i$. Hence
\[
    X_i = A_i N_1 s + B_i N_2 (Ks+Mt) \equiv 0 \pmod p.
\]
Thus in all cases $p\mid X_i$.

\smallskip
\emph{(B) Let $p$ be a prime divisor of $N/N_P(x_i)$. We prove that $p\nmid X_i$.}

\smallskip
\noindent
If $p\mid M$, then by definition of $N_P(x_i)$ we have $p\nmid x_i,y_i$, hence
\[
0 \not\equiv A_i x_0 + B_i y_0
   = A_i\!\left(\frac{x_0}{N_1}\right)\!N_1 + B_i\!\left(\frac{y_0}{N_2}\right)\!N_2
   \equiv (A_i N_1 + B_i N_2 K)\!\left(\frac{x_0}{N_1}\right) \pmod p,
\]
so $A_i N_1 + B_i N_2 K \not\equiv 0 \pmod p$.
Therefore,
\[
X_i = A_i N_1 s + B_i N_2 (Ks+Mt)
   \equiv (A_i N_1 + B_i N_2 K)s \not\equiv 0 \pmod p.
\]

\smallskip
\noindent
If $p\nmid M$, then $p$ divides either $N_1$ or $N_2$.
Suppose $p\mid N_1$ (the other case is analogous).  
Since $p\nmid N_P(x_i)$, we have $x_i = A_i x_0 + B_i y_0 \equiv B_i y_0 \pmod p$,
so $p\nmid B_i$.  
Using $\gcd(x_\ell,y_\ell,N)=1$ we deduce
\[
X_i = A_i N_1 s + B_i N_2 (Ks+Mt) \not\equiv 0 \pmod p.
\]
Thus, in all cases $p\nmid X_i$.

\smallskip
From (A) and (B) we conclude $N_P(x_i)=N_P(X_i)$ for all $i$; the same argument gives $N_P(y_i)=N_P(Y_i)$.
By the rule in \cref{def1.3}, this implies
\[
  (X_0,Y_0)\to (X_1,Y_1)\to\cdots\to (X_\ell,Y_\ell),\qquad (X_i,Y_i)\in\APT_P(x,y).
\]
Since $\ell\ge1$, the matrix $M_\ell\in \langle L,R\rangle$ is not the identity.
Therefore, for any $n\in\N$,
\[
    (x_0, y_0)M_\ell^{\,n}\in\APT_P(x,y),
\]
and these pairs are all distinct.
Hence $\APT_P(x,y)$ is infinite, i.e.\ $P$ is of infinite type.
\end{proof}

Using the criterion above, we can construct the following example of an infinite type set.

\begin{example} \label{ex2.3}
    Let $P = \{ 2,3,5,7,19 \}$. Then $|\APT_P(2,325)| = \infty$, and hence $P$ is of infinite type.
\end{example}

\begin{proof}
    Let
    \begin{align*}
    \scalebox{0.95}{$
        A = \{(2, 325),(2, 327),(329, 327),(329, 656),(329, 985),(1314, 985),(2299, 985),(3284, 985)\}.
        $}
    \end{align*}
    Then $A \subset \APT_P(2,325)$, and the following sequence holds:
    \begin{align*}
    \scalebox{0.85}{$
        (2, 325) \rightarrow (2, 327) \rightarrow (329, 327) \rightarrow (329, 656) \rightarrow (329, 985)
        \rightarrow (1314, 985) \rightarrow (2299, 985) \rightarrow (3284, 985).
        $}
    \end{align*}
    Moreover,
    \[
    (3284,985) = (2s, 5(65s + 399t)), \quad
    N_P(2) = N_P(3284), \quad N_P(325) = N_P(985)
    \]
    for $(s,t) = (1642,-267)$.  
    By \cref{prop2.2}, $P$ is of infinite type.
\end{proof}

The following proposition shows that the property of being of infinite type is preserved under inclusion.

\begin{proposition}\label{prop2.4}
Let $P$ and $Q$ be finite subsets of $\Pri$ with $P \subset Q$.  
If $P$ is of infinite type, then $Q$ is also of infinite type.
\end{proposition}

\begin{proof}
Since $P$ is of infinite type, there exists $(x,y) \in \N^2$ such that $|\APT_P(x,y)| = \infty$.  
That is, there exists a sequence $((x_i, y_i))_{i \ge 0} \subset \APT_P(x,y)$ with $(x_0,y_0) = (x,y)$ and $(x_i,y_i) \to (x_{i+1},y_{i+1})$ for all $i \ge 0$.

Let $v = \gcd(x, y, N(Q))$.  
Then $\gcd(x/v,\, y/v,\, N(Q)) = 1$, and $(x_i/v,\, y_i/v) \in \N^2$ for all $i \ge 0$ because $v$ divides both $x_i$ and $y_i$.  
Consider $\APT_Q(x/v, y/v)$.  
By definition, $(x/v, y/v) \in \APT_Q(x/v, y/v)$.

For each $i \ge 0$, since $(x_i, y_i) \in \APT_P(x,y)$, there exist primes $p, p' \in P$ such that $p \mid x_i$ and $p' \mid y_i$.  
Because $P \subset Q$, the same divisibility conditions hold in $\APT_Q(x/v, y/v)$.  
Hence, by induction on $i$, we have $(x_i/v, y_i/v) \in \APT_Q(x/v, y/v)$ for all $i \ge 0$.

Therefore, $\APT_Q(x/v, y/v)$ is infinite, and thus $Q$ is of infinite type.
\end{proof}

From \cref{prop2.4}, the property of being of infinite type is preserved under inclusion.  
Therefore, it is natural to consider minimal subsets $P$ that are of infinite type.  
We call such minimal sets \emph{primitive sets}.

\begin{definition}\label{def2.5}
Let $P$ be a finite subset of $\Pri$.  
We say that $P$ is a \emph{primitive set} if it is of infinite type and every proper subset of $P$ is of finite type.  
We denote by $\Prim_r$ the set of all primitive sets $P$ with $|P| = r$, and by $\Prim$ the set of all primitive sets.
\end{definition}

From \cref{ex2.3} and \cref{main} (1), we see that $P = \{2, 3, 5, 7, 19\}$ is a primitive set.  
The following table lists the number of primitive sets.  
Using a computer search, we completely classified all primitive sets with $|P| \le 6$. In addition, our computation has so far identified $1095$ primitive sets in total.

\newpage

\renewcommand{\arraystretch}{1.1}
\begin{table}[h]
    \centering
    \caption{Number of primitive sets}
    \begin{tabular}{|c|c|l|}
    \hline
    $r$ & $|\Prim_{r}|$ & Example of $P \in \Prim_{r}$ \\
    \hline\hline
    $\leq4$ & $0$ & $\emptyset$ \\
    \hline
    $5$ & $2$ & $\{ 2,3,5,13,17 \}$ \\
    \hline
    $6$ & $27$ & $\{2,3,5,7,11,13\}$ \\
    \hline
    $7$ & $\geq 488$ & $\{3, 7, 11, 17, 19, 23, 29\}$ \\
    \hline
    $8$ & $\geq 395$ & $\{3, 5, 13, 19, 31, 43, 79, 109\}$ \\
    \hline
    $9$ & $\geq 183$ & $\{5, 7, 29, 41, 59, 61, 71, 101, 109\}$ \\
    \hline
    \end{tabular}
\end{table}

To classify primitive sets, we introduce an algorithmic procedure that
systematically explores all possible assignments of correspondences and pairings.
The main routine, Algorithm~1, generates candidates for primitive
sets by iteratively applying Algorithm~2.
Starting from an initial correspondence in which every prime is mapped to~$1$,
the algorithm expands all admissible paths up to a prescribed depth~$limit$.
The remaining states after this expansion provide potential primitive candidates.
Since for a fixed cardinality $|P|$ there exist only finitely many primitive sets,
this procedure terminates after finitely many steps.

\begin{breakalg}{1: Generate Candidates of Primitive Sets}
\begin{algorithmic}[1]

\Require $\mathrm{num},\; limit$
\Ensure A list $\mathcal{L}$ of remaining states after $limit$ expansions

\State Construct $P_{\mathrm{cor}} := \{p_1 \mapsto 1, \dots, p_{\mathrm{num}} \mapsto 1\}$
\State Set $\mathrm{Path} := [(a,b)]$

\State Generate all admissible initial assignments for $P_{\mathrm{pair}}$

\For{each initial triple $T = (P_{\mathrm{cor}}, P_{\mathrm{pair}}, \mathrm{Path})$}
    \State $\mathcal{L} \gets \{T\}$
    \For{$i = 1$ to $limit$}
        \State Replace $\mathcal{L}$ with the list obtained by applying Algorithm 2 to all elements of $\mathcal{L}$
    \EndFor
    \State \Return $\mathcal{L}$
\EndFor

\end{algorithmic}
\end{breakalg}

The local update rule used in Algorithm 1 is described in
Algorithm 2.  
Given the current correspondence, pairing, and path, the algorithm produces all
admissible next states in the search tree.  
Depending on whether $a+b$ (the next branching value) is divisible by a prime in
$P$, the correspondence or pairing may remain unchanged, or may be updated by
assigning unused primes.  
This local rule governs the branching behavior of the paths generated in
Algorithm 1.

\begin{breakalg}{2: Generate new paths}
\begin{algorithmic}[1]

\Require $P_{\mathrm{cor}},\; P_{\mathrm{pair}},\; \mathrm{Path}$
\Ensure A list of triples $(\text{new\_}P_{\mathrm{cor}},\; \text{new\_}P_{\mathrm{pair}},\; \text{new\_Path})$

\State Let $(a,b) \gets \mathrm{Path}[-1]$
\State Initialize an empty list $\mathcal{L}$

\If{$a+b$ is divisible by some $p \in P$}
    \State $\text{new\_}P_{\mathrm{cor}} \gets P_{\mathrm{cor}}$
    \State $\text{new\_}P_{\mathrm{pair}} \gets P_{\mathrm{pair}}$
    \State Append $(\text{new\_}P_{\mathrm{cor}}, \text{new\_}P_{\mathrm{pair}},
           \mathrm{Path} + [(a,a+b)])$ to $\mathcal{L}$
    \State Append $(\text{new\_}P_{\mathrm{cor}}, \text{new\_}P_{\mathrm{pair}},
           \mathrm{Path} + [(a+b,b)])$ to $\mathcal{L}$
\Else
    \State (i) use the current pairing
    \State Update $P_{\mathrm{cor}}$ so that $a+b$ is divisible by some $p\in P$
           according to $P_{\mathrm{pair}}$
    \State Append the two updated paths to $\mathcal{L}$

    \State (ii) use new pairings with unused primes
    \For{each admissible new pairing $P_{\mathrm{pair}}'$ using unused primes in $P$}
        \State Append the two extended paths to $\mathcal{L}$
    \EndFor
\EndIf

\State \Return $\mathcal{L}$

\end{algorithmic}
\end{breakalg}

The complete implementation of both algorithms is available at the following GitHub repository: \url{https://github.com/yusukekusuy/additive_prime_tree}.

To illustrate how the procedure works in practice, we now present an explicit
example in a specific case.

\begin{example}
For $|P| = 3$, we start with
\[
P_{\mathrm{cor}} = \{p_{1}\!:\!1,\; p_{2}\!:\!1,\; p_{3}\!:\!1\}, 
\qquad
\mathrm{Path} = [(a,b)].
\]
All admissible initial assignments of $P_{\mathrm{pair}}$ produce the following
two triples:
\[
T_{1} = 
\bigl(
P_{\mathrm{cor}},\;
\{\, a:\{p_{1}\},\; b:\{p_{2}\} \,\},\;
[(a,b)]
\bigr),
\]
\[
T_{2} = 
\bigl(
P_{\mathrm{cor}},\;
\{\, a:\{p_{1},p_{2}\},\; b:\{p_{3}\} \,\},\;
[(a,b)]
\bigr).
\]
These two triples serve as the initial inputs to Algorithm 1 .
\end{example}

\begin{example}
\noindent\textbf{Case 1.}
Let
\[
P_{\mathrm{cor}} = \{p_{1}\!:\!1,\; p_{2}\!:\!1,\; p_{3}\!:\!1,\; p_{4}\!:\!1\},
\quad
P_{\mathrm{pair}} = \{\, a:\{p_{1}\},\; b:\{p_{2}\} \,\},
\quad
\mathrm{Path} = [(a,b)],
\]
and set
\[
T = (P_{\mathrm{cor}}, P_{\mathrm{pair}}, \mathrm{Path}).
\]
Then $\mathrm{Path}[-1] = (a,b)$ and the sum is $a+b$.  
By construction of $P_{\mathrm{pair}}$, the value $a+b$ is not divisible by
$p_{1}$ or $p_{2}$, so in order for the tree to continue branching we must assign
a new prime divisor to $a+b$.

We consider the two extensions
\[
\mathrm{Path}_{1} = [(a,b), (a,a+b)],
\qquad
\mathrm{Path}_{2} = [(a,b), (a+b,b)].
\]
Using the unused primes $p_{3}$ and $p_{4}$, we introduce the following new
pairings:
\[
P_{\mathrm{pair}}^{(1)} =
\{\, a:\{p_{1}\},\; b:\{p_{2}\},\; a+b:\{p_{3}\} \,\},
\]
\[
P_{\mathrm{pair}}^{(2)} =
\{\, a:\{p_{1}\},\; b:\{p_{2}\},\; a+b:\{p_{3},p_{4}\} \,\}.
\]
Accordingly, Algorithm 2 produces the four new triples
\begin{align*}
    T_{1} = (P_{\mathrm{cor}}, P_{\mathrm{pair}}^{(1)}, \mathrm{Path}_{1}),\quad
    T_{2} = (P_{\mathrm{cor}}, P_{\mathrm{pair}}^{(1)}, \mathrm{Path}_{2}),\\
    T_{3} = (P_{\mathrm{cor}}, P_{\mathrm{pair}}^{(2)}, \mathrm{Path}_{1}),\quad
    T_{4} = (P_{\mathrm{cor}}, P_{\mathrm{pair}}^{(2)}, \mathrm{Path}_{2}).
\end{align*}

\noindent\textbf{Case 2.}
Let
\[
P_{\mathrm{cor}} = \{p_{1}\!:\!1,\; p_{2}\!:\!1,\; p_{3}\!:\!1,\; p_{4}\!:\!1\},
\quad
P_{\mathrm{pair}} = 
\{\, a:\{p_{1}\},\; b:\{p_{2}\},\; a+b:\{p_{3},p_{4}\} \,\},
\]
and
\[
\mathrm{Path} = [(a,b),\,(a,a+b)],
\qquad
T = (P_{\mathrm{cor}}, P_{\mathrm{pair}}, \mathrm{Path}).
\]

Here $\mathrm{Path}[-1] = (a, a+b)$, so the next sum is
\[
a + (a+b) = 2a + b.
\]
In order for the tree to continue branching, $2a+b$ must be divisible by
some $p \in P$.  
From the pairing structure $P_{\mathrm{pair}}$, this happens only when
\[
p_{2} = 2.
\]

Thus $P_{\mathrm{cor}}$ must be updated by assigning $p_{2}$ to $2a+b$, giving
\[
P_{\mathrm{cor}}' =
\{\, p_{1}\!:\!1,\; p_{2}\!:\!2,\; p_{3}\!:\!1,\; p_{4}:1 \,\}.
\]

We then obtain the two extensions
\[
\mathrm{Path}_{1}
= [(a,b),\,(a,a+b),\,(2a+b,b)],
\quad
\mathrm{Path}_{2}
= [(a,b),\,(a,a+b),\,(a,2a+b)].
\]

Accordingly, Algorithm 2 produces the two new triples
\[
T_{1} = (P_{\mathrm{cor}}',\, P_{\mathrm{pair}},\, \mathrm{Path}_{1}),
\qquad
T_{2} = (P_{\mathrm{cor}}',\, P_{\mathrm{pair}},\, \mathrm{Path}_{2}).
\]
\end{example}

\section{Proof of \cref{main}}
In this section, we prove \cref{main}.  
Using Algorithm 1, we see that any finite subset $P \subset \Pri$ with $|P| \le 4$ is of finite type.  
For completeness, we provide an explicit enumeration of all cases in the GitHub repository under \texttt{additive\_prime\_tree/figure}, since the number of case distinctions is too large to include in the Appendix.\\ 

Next, we prove \cref{main} (2).  
For a finite type set $P$ and $x, y, r \in \N$, we define the quantities $L(P;x,y)$, $L(P)$, $L_r$, $C(P;x,y)$, $C(P)$ and $C_{r}$ by
\begin{align*}
    L(P;x,y) &:= \max \{ \ell \in \N \mid ((x_i,y_i))_{0 \le i \le \ell} \subset \APT_P(x,y)
        \text{ with } (x_0,y_0) \to \cdots \to (x_\ell,y_\ell) \},\\
    L(P) &:= \sup_{(x,y) \in \N^2} L(P;x,y), \qquad
    L_r := \sup_{\substack{P \text{ of finite type} \\ |P| = r}} L(P),\\
    C(P;x,y) &:= \max \{ \ell \in \N \mid ((x, i x + y))_{0 \le i \le \ell} \subset \APT_P(x,y) \},\\
    C(P) &:= \sup_{(x,y) \in \N^2} C(P;x,y), \qquad
    C_r := \sup_{\substack{P \text{ of finite type} \\ |P| = r}} C(P).
\end{align*}

To illustrate these definitions, we consider the following example.

\begin{example}
    From \cref{ex1.4} and the data in \texttt{additive\_prime\_tree/figure/case\_|P|=3}, we have
    \begin{align*}
        L(\{ 2,3,5 \};2,3) &= 3,\quad L(\{ 2,3,5 \}) = 3,\quad L_{3} = 3,\\ 
        C(\{ 2,3,5 \};2,3) &= 2,\quad C(\{ 2,3,5 \}) = 3,\quad C_3 = 3.
    \end{align*}
\end{example}

We prove the following lemma and proposition to establish \cref{main} (2).

\begin{lemma}\label{lemma3.2}
For any $r \in \N$, we have $C_r < \infty$.
\end{lemma}

\begin{proof}
Let $P$ be a finite subset of $\Pri$ with $|P| = r$, and let $x, y \in \N$ satisfy $\gcd(x, y, N(P)) = 1$ and $|\APT_P(x,y)| > 1$. 
If $p \in P$ divides $x$, then $p \nmid (i x + y)$ for all $i \ge 0$.  
Set
\[
    P' := \{\, p \in P \mid p \nmid x \,\}.
\]
For each $p \in P'$, the element $x$ is invertible modulo $p$, hence
\[
    p \mid (i x + y) \ \Longleftrightarrow\ i \equiv -x^{-1} y \pmod{p}.
\]
Define
\[
    U(P; x, y) := \bigcup_{p \in P'} \bigl\{\, i \in \Z_{\ge 0} \mid i \equiv -x^{-1} y \pmod{p} \,\bigr\}.
\]
Then the condition $\{0,1,\cdots,\ell\} \subset U(P;x,y)$ means that every $i \in \{0,\cdots,\ell\}$ lies in one fixed residue class modulo some $p \in P'$, i.e., a covering condition analogous to the Jacobsthal problem~\cite{Erd}.

We make this precise via a shifted-sieve lower bound.  
Let $Q := N(P')$, and for each $p \in P'$ put $a_p \equiv -x^{-1} y \pmod{p}$.  
For any squarefree $d \mid Q$, let $a_d \pmod{d}$ denote the lift given by the Chinese Remainder Theorem of the congruences $a_d \equiv a_p \pmod{p}$ for all $p \mid d$.  
For an interval $\mathcal{A} = \{m, m+1, \dots, m+X-1\}$ define
\begin{align*}
    A_d &:= \bigl|\{\, i \in \mathcal{A} \mid i \equiv a_d \pmod{d} \,\}\bigr|,\\
    S(\mathcal{A}, Q) &:= \bigl|\{\, i \in \mathcal{A} \mid i \not\equiv a_p \pmod{p} \text{ for all } p \mid Q \,\}\bigr|
    = \sum_{d \mid Q} \mu(d) A_d.
\end{align*}
Then $\bigl| A_d - X/d \bigr| \le 1$ holds for all $d \mid Q$, and $S(\mathcal{A}, Q)$ counts the
\emph{survivors} $i \in \mathcal{A}$ that avoid all forbidden residue classes $i \equiv a_p \pmod{p}$ for $p \mid Q$.

Fix $z\ge2$ to be the $|P'|$-th prime number and let $X \asymp z^2$.
By Iwaniec’s shifted sieve (Lemma~1 and Lemma~2 in~\cite{Iwa}), we have
\[
  S(\mathcal A,Q)>0,
\]
hence there exists $i\in\mathcal A$ such that
\[
  i\not\equiv a_p \pmod p \qquad\text{for all }p\in P'.
\]
Since $z\sim |P'|\log|P'|$, we obtain
\[
  X \asymp |P'|^2(\log|P'|)^2.
\]
Therefore any block $\{0,1,\dots,\ell\}\subset U(P;x,y)$ must satisfy
\[
  \ell+1 \;\le\; c\,|P'|^2(\log|P'|)^2 \;<\; c\,r^2(\log r)^2
\]
for some absolute constant $c>0$, as claimed.
\end{proof}

Having established the finiteness of $C_{r}$, we next prove the finiteness of $L_{r}$, which measures the maximal possible path length in the additive prime tree.

\begin{lemma}\label{lemma3.3}
For any $r \in \N$, we have $L_r < \infty$.
\end{lemma}

\begin{proof}
We argue by induction on $r$. The cases $r\le 4$ follow from Algorithm 1 (and the enumeration in \texttt{additive\_prime\_tree/figure/case\_|P|=4}), hence $L_r<\infty$ for $r\le 4$.

Assume inductively that $L_r<\infty$ and fix a prime
$p\geq F_{L_r+C_r+1}+1$, where $(F_k)_{k \in \Z_{\geq 0}}$ denotes the Fibonacci numbers (with $F_0=F_1=1$, so $F_{k+1}=F_k+F_{k-1}$).
Let $P\subset\Pri$ be of finite type with $|P|=r$ and $p\in \Pri \backslash P$. We claim that
\[
  L(P\cup\{p\}) \le 2L_r + C_r.
\]
Suppose, for contradiction, that $L(P\cup\{p\}) > 2L_r + C_r$.
Then there exist $x,y\in\N$ and a path $((x_i,y_i))_{0\le i\le \ell}\subset \APT_{P\cup\{p\}}(x,y)$ with
$(x_0,y_0)\to\cdots\to(x_\ell,y_\ell)$ and $\ell\ge 2L_r+C_r+1$.
Since $L(P)<\infty$ and $|P|=r$, there must be some $f\le L_r$ such that $p\mid x_f$ or $p\mid y_f$; otherwise, the first $L_r+1$ steps would already lie in $\APT_P(x,y)$, forcing $L(P)\ge L_r+1$, a contradiction. Without loss of generality, assume $p\mid x_f$.

For $j\ge 0$ we can write
\[
  (x_{f+j}, y_{f+j})
  =
  (x_f,y_f)\begin{pmatrix} A_j & D_j \\ B_j & E_j \end{pmatrix}
  \qquad A_j,B_j,D_j,E_j\in\Z_{\ge0},
\]
where the step matrices are products of $\bigl(\begin{smallmatrix}1&0\\1&1\end{smallmatrix}\bigr)$ and
$\bigl(\begin{smallmatrix}1&1\\0&1\end{smallmatrix}\bigr)$. A routine induction gives the Fibonacci bounds
$A_j,B_j,D_j,E_j \le F_j$.
Define
\[
  f'_1 := \max\{\, j\in\{0,1,\dots,\ell-f\} \mid B_j=0 \,\}.
\]
By \cref{lemma3.2}, we have $f_{1}'\le C_r$, and furthermore
$p\nmid x_{f+f'_1+1}$ and $p\nmid y_{f+f'_1+1}$.
By the choice of $f'_1$, we have $B_j>0$ for all $j\ge f'_1+1$; similarly, from the step recurrences one checks that $E_j>0$ for all $j\ge f'_2+1$.
Since
\[
  p \geq F_{L_r+C_r+1}+1 \ >\ F_{L_r+f'+1} \ \ge\ \max\{B_j,E_j\}
  \qquad (f'+1 \le j \le L_r+f'+1),
\]
where $f' = \max\{f'_1, f'_2\}$, we obtain for such $j$ that
\[
  \gcd(A_j x_f + B_j y_f,\, p) = \gcd(B_j y_f,\, p)=1,\quad
  \gcd(D_j x_f + E_j y_f,\, p) = \gcd(E_j y_f,\, p)=1,
\]
using $p\mid x_f$ and $p\nmid y_{f+f'+1}$ noted above. Hence the segment
\[
  \bigl((x_{f+f'+1+s},y_{f+f'+1+s})\bigr)_{0\le s\le L_r}
  \subset \APT_{P}(x,y).
\]
Therefore
\[
  L_r+1 \le L(P;x,y) \le L(P) \le L_r,
\]
a contradiction. We conclude that $L(P\cup\{p\}) \le 2L_r+C_r$.

Now set
\begin{align*}
  S &:= \{\, Q\subset\Pri \mid |Q|=r+1,\ Q \text{ of finite type}\,\},\\
  S_1 &:= \{\, Q\in S \mid \max Q \geq F_{L_r+C_r+1}+1 \,\},\qquad
  S_2 := S\setminus S_1.
\end{align*}
Then
\[
  L_{r+1}
  = \sup_{Q\in S} L(Q)
  = \max\Big\{ \sup_{Q\in S_1} L(Q),\ \max_{Q\in S_2} L(Q) \Big\}
  \le \max\Big\{ 2L_r+C_r,\ \max_{Q\in S_2} L(Q) \Big\}.
\]
Since $S_2$ is finite, $\max_{Q\in S_2}L(Q)<\infty$, and hence $L_{r+1}<\infty$.
This completes the induction and proves $L_r<\infty$ for all $r\in\N$.
\end{proof}

Let $P$ be a finite subset of $\Pri$ such that every proper subset of $P$ is of finite type and $|P| = r \geq 4$. 
From the proof of \cref{lemma3.3}, $P$ is not a primitive set if $p > F_{L_{r-1}+C_{r-1}+1}+1$ for some $p \in P$. 
Therefore, there exist only finitely many primitive sets $P$ with $|P| = r$. 
Hence \cref{main} (2) is proved. \\

We now proceed to the proof of \cref{main} (3). 
By \cref{prop2.4}, the property of being of infinite type is preserved under inclusion; in particular, 
if a finite subset $Q \subset \Pri$ contains some primitive set $P$, then $Q$ is of infinite type. 
Equivalently, $Q$ is of finite type if and only if it contains no primitive set.

For a finite subset $P \subset \Pri$, define
\begin{align*}
    \mC(P) 
    &:= \lim_{N \to \infty} \frac{1}{N}\,\big|\{\, M \le N \mid P(M) = P \,\}\big| \\
    &= \bigg(\prod_{p \in P} \frac{1}{p^2}\bigg)\bigg(\prod_{p \notin P}\Big(1-\frac{1}{p^2}\Big)\bigg)
     \;=\; \frac{6}{\pi^2}\,\prod_{p \in P}\frac{1}{p^2-1}.
\end{align*}
Then
\[
    \lim_{N \to \infty} \frac{|\mA(N)|}{N} \;=\; 1 \;-\; \sum_{P:\,\text{infinite type}} \mC(P).
\]

Since a complete list of primitive sets is out of reach, we derive upper and lower bounds. 
For the upper bound we use that every infinite-type set contains a primitive set:
\begin{align*}
    \lim_{N \to \infty} \frac{|\mA(N)|}{N} 
    \;\le\; 1 - \sum_{P \in \Prim} \mC(P) 
    \;\le\; 1 - \sum_{\substack{P \in \Prim\\ \text{known}}} \mC(P)
    \;=\; 0.99999991529\cdots
\end{align*}
For the lower bound,
since all primitive sets with $|P|\le 6$ have been completely classified,
and those with $|P|=7$ and $3\notin P$ are also classified,
we may incorporate these contributions explicitly.
Moreover, taking into account that $3\in P$ necessarily holds
for primitive sets with $|P|=5$, we obtain
\begin{align*}
    \lim_{N \to \infty} \frac{|\mathcal{A}(N)|}{N}
    &\geq \sum_{\substack{P \subset \Pri\\ |P| \le 6}} \mathcal{C}(P)
      + \sum_{\substack{P \subset \Pri\\ |P| = 7\\ 3 \notin P}} \mathcal{C}(P)
      - \sum_{P \in \Prim_{5} \cup \Prim_{6}} \mathcal{C}(P) \\
    &\quad
      - \sum_{\substack{P \in \Prim_{7}\\ 3 \notin P}} \mathcal{C}(P)
      - \sum_{P \in \Prim_{5}}\ \sum_{p \in \Pri\backslash P} \mathcal{C}(P \cup \{p\})
      - \sum_{\substack{P \in \Prim_{6}\\ 3 \notin P}}
        \sum_{p \in \Pri\backslash(P \cup \{3\})} \mathcal{C}(P\cup\{p\}) \\
    &= 0.99999991254\cdots.
\end{align*}

Hence we obtain \cref{main} (3). \hfill$\square$

\section{Application} \label{app}

In this section, we investigate the connection between the finite-type property and the structure of modular forms. We follow the notation and framework of Dickson and Neururer~\cite{DN}, and prove \cref{thm1.6}, which generalizes their Theorems~1.2 and~1.3.
We begin by recalling the relevant definitions and results from~\cite{DN}.

We define the slash operator by
\[
    \bigg(f\big|_{k} 
    \begin{pmatrix} a & b \\ c & d \end{pmatrix}
    \bigg)(z)
    := (ad-bc)^{k/2}(cz+d)^{-k}\,
    f\!\left(\frac{az+b}{cz+d}\right).
\]
We will write simply $\mid$ for $\mid_{k}$ when the weight $k$ is clear from context.

We consider the Eisenstein series
\begin{align*}
    E_{k}^{\phi,\psi}(z)
    &= e_{k}^{\phi,\psi}
    + 2\sum_{n\ge1}\sigma_{k,\phi,\psi}(n)q^{n}
    \in \mM_{k}(\Gamma_{0}(N),\phi\psi),
\end{align*}
where $q=e^{2\pi i z}$, and $\phi,\psi$ are primitive Dirichlet characters of levels $N_{1}$ and $N_{2}$ with $N_{1}N_{2}=N$.  
The coefficients are given by
\[
    \sigma_{k,\phi,\psi}(n)
    = \sum_{d\mid n}\phi(n/d)\psi(d)d^{k-1},
\]
and the constant term is
\[
    e_{k}^{\phi,\psi} =
    \begin{cases}
        L(\psi,1-k), & N_{1}=1,\\[4pt]
        L(\phi,0), & N_{2}=1,\ k=1,\\[4pt]
        0, & \text{otherwise.}
    \end{cases}
\]

The space $\mM_{k}(\Gamma_{0}(N))$ decomposes as a direct sum
\[
    \mM_{k}(\Gamma_{0}(N))
    = \mS_{k}(\Gamma_{0}(N))
      \oplus
      \mE_{k}(\Gamma_{0}(N)),
\]
with respect to the Petersson inner product,  
where $\mS_{k}(\Gamma_{0}(N))$ denotes the space of cusp forms and $\mE_{k}(\Gamma_{0}(N))$ denotes the Eisenstein subspace.

Let $B_{d}$ denote the lifting operator defined by
\[
    (f|B_{d})(z) = d^{k/2}f(dz).
\]
For even $k\ge2$, we define $\mQ_{k}(N)\subset\mM_{k}(\Gamma_{0}(N))$ to be the subspace generated by all products of the form
\[
    E_{\ell}^{\phi,\psi}|B_{d_{1}d}
    \;\cdot\;
    E_{k-\ell}^{\bar\phi,\bar\psi}|B_{d_{2}d},
\]
which lie in $\mM_{k}(\Gamma_{0}(N))$ under the additional condition that $d_{1}M_{1}$ divides the squarefree part of $N$.  
Here $\ell\in\{1,\dots,k-1\}$, $\phi$ and $\psi$ are primitive characters modulo $M_{1}$ and $M_{2}$ satisfying $\phi\psi(-1)=(-1)^{\ell}$,  
and $d_{1},d_{2}\in\N$ satisfy $d_{1}M_{1}d_{2}M_{2}d\mid N$.  
We exclude the cases $(\phi,\psi,\ell)=(\bo,\bo,2)$ and $(\bo,\bo,k-2)$, where $\bo$ denotes the trivial character, since these choices do not yield modular forms.  
The condition on $d_{1}M_{1}$ implies that $\gcd(d_{1}M_{1},d_{2}M_{2}d)=1$. Dickson and Neururer proved \Cref{thm1.1,thm1.2}.

Let us now give a proof of \cref{thm1.6}. 
Let $N \in \N$. 
For each divisor $M \mid N$, denote by $D(M)$ the set of primitive Dirichlet characters modulo $M$. 
Define
\[
    B(N) \subset \bigsqcup_{M \mid N} D(M) \times \{1,2,\dots,k-1\}
\]
to be the set of pairs $(\alpha, \ell)$ satisfying
\begin{align*}
    \alpha(-1) &= (-1)^{\ell},\\
    (\alpha, \ell) &\neq (\bo,2),\ (\bo,k-2).
\end{align*}

We define $P_{k}(N) \subset \mM_{k}(\Gamma_{0}(N))$ to be the subspace generated by the products
\[
    \big(E_{\ell}^{\bo, \alpha}\,E_{k-\ell}^{\bo, \overline{\alpha_{N}}}\big)\big|W_{S}^{N},
\]
for all $(\alpha,\ell)\in B(N)$ and all subsets $S$ of the set of prime divisors of the squarefree part of $N$. 
Here $\alpha_{N}$ denotes the extension of $\alpha$ to a character modulo $N$, 
and $W_{S}^{N}$ denotes the Atkin–Lehner involution associated to $S$.

Dickson and Neururer proved that the projection of $P_k(N)$ onto the cuspidal subspace coincides with $\mS_{k}^{\mathrm{new}}(\Gamma_{0}(N))$.
To establish this, they employed several criteria relating the vanishing of twisted $L$-values of a modular form $f$ to the vanishing of $f$ itself. Let $f \in \mS_{k}(\Gamma_{0}(N), \chi)$ with Fourier expansion 
\[
    f(z) = \sum_{n \ge 1} a_n e^{2\pi i n z},
\]
and let $\alpha$ be a Dirichlet character modulo $M$. 
We define the twist of $f$ by $\alpha$ as
\[
    f_{\alpha}(z) := \sum_{n \ge 1} a_n \alpha(n) e^{2\pi i n z}.
\]

We recall some facts from the theory of modular symbols 
(for details, see \cite{Mer} and \cite{Ste}). 
The space $\M_{k}(\Gamma_{1}(N))$ of modular symbols is generated by the Manin symbols $[P,g]$, 
where $P$ is a homogeneous polynomial in $\C[X,Y]$ of degree $k-2$ and $g \in \SL_{2}(\Z)$. 
The symbol $[P,g]$ depends only on $P$ and the coset $\Gamma_{1}(N)g$. 
Mapping a matrix $g$ to its bottom row modulo $N$ induces a bijection
\[
    \Gamma_{1}(N)\backslash \SL_{2}(\Z)
    \;\longleftrightarrow\;
    E_{N} := \{(u,v)\in (\Z/N\Z)^{2} \mid \gcd(u,v,N)=1\}.
\]
Note that $\gcd(u,v,N)$ is well defined, i.e., independent of the choice of representatives for $u,v$. 
For $(u,v)\in E_N$, we write
\[
    [P,(u,v)] := [P,g],
\]
for any $g\in\SL_{2}(\Z)$ whose bottom row is congruent to $(u,v)$ modulo $N$. For $f\in\mS_{k}(\Gamma_{1}(N))$, we define
\[
    \xi_{f}([P,g])
    := \int_{g\cdot 0}^{g\cdot \infty} f(z)\,(gP)(z,1)\,dz,
\]
where $g$ acts on $P\in\C[X,Y]$ by
\[
    (gP)(X,Y) := P\big(g^{-1}(X,Y)^{{T}}\big).
\]

Using the generators $[X^{j}Y^{k-2-j},(u,v)]$ with $0\le j\le k-2$ and $(u,v)\in E_N$, 
we write
\[
    \xi_{f}(j;u,v)
    := \xi_{f}\big([X^{j}Y^{k-2-j},(u,v)]\big).
\]
By Theorem~8.4 of \cite{Ste}, the following relations hold:
\begin{align}
    &\xi_{f}(j;u,v) + (-1)^{j}\,\xi_{f}(k-2-j;v,-u) = 0,\label{(4.1)}\\[4pt]
    &\xi_{f}(j;u,v)
    + \sum_{i=0}^{k-2-j} (-1)^{k-2-i}\binom{k-2-j}{i}\,\xi_{f}(i;v,-u-v)
    &\notag\\
    & \qquad+ \sum_{i=k-2-j}^{k-2} (-1)^{i}\binom{j}{i-k+2+j}\,\xi_{f}(i;-u-v,u)
    = 0, \label{(4.2)}\\[4pt]
    &\xi_{f}(j;u,v) - (-1)^{k-2}\,\xi_{f}(j;-u,-v) = 0.\label{(4.3)}
\end{align}

We also define the $\xi_{f}^{\pm}$ by
\[
    \xi_{f}^{\pm}(j;u,v)
    := \frac{1}{2}\bigl(\xi_{f}(j;u,v) \pm (-1)^{j+1}\xi_{f}(j;-u,v)\bigr).
\]
Applying $\xi_{f}$ to the Manin relations and taking suitable linear combinations, 
we see that (\ref{(4.1)})–(\ref{(4.3)}) remain valid when $\xi_{f}$ is replaced by $\xi_{f}^{\pm}$.

By Proposition~8 of \cite{Mer}, 
the maps $f \mapsto \xi_{f}^{+}$ and $f \mapsto \xi_{f}^{-}$ are injective; 
hence $f$ vanishes if and only if all $\xi_{f}^{\pm}$ vanish.
The values $\xi_{f}(j;u,v)$ are closely related to the critical values of the $L$-function of $f$:  
indeed, for $g = \bigl(\begin{smallmatrix} a & b \\ c & d \end{smallmatrix}\bigr) \in \SL_{2}(\Z)$ 
with $(c,d)\equiv(u,v)\pmod N$, we have
\[
    \xi_{f}(j;u,v)
    = \frac{j!}{(-2\pi i)^{j+1}}\,L(f|g, j+1).
\]

Dickson and Neururer proved the following theorems. In particular, \cref{thm4.2} is closely related to the finiteness of additive prime trees (APT).

\begin{theorem}[\cite{DN}, Thm.~3.2]\label{thm4.1}
Let $\epsilon \in \{0,1\}$, $N \in \N$, $k \ge 2$, and let $f \in \mS_{k}^{\mathrm{new}}(\Gamma_{0}(N))$ be an eigenfunction of all Atkin–Lehner operators $W_{p}^{N}$ for $p \in T$, where $T$ is the set of primes with $v_{p}(N)=1$. 
Assume that
\[
L\big(f_{\alpha}\big|W_{S}^{NM},\, j+1\big)=0
\]
for every primitive character $\alpha \pmod M$ with $M\mid N$, every $0\le j\le k-2$ satisfying $\alpha(-1)=(-1)^{j+\epsilon}$, and every subset $S\subset \overline{T}$ of prime divisors of $N/M$, where $\overline{T}$ is the set of primes with $v_{p}(N)\geq 2$. 
Then $f=0$.
\end{theorem}

\begin{proof}
We treat the case $\epsilon=1$; the case $\epsilon=0$ is analogous with $\xi_f^{-}$ in place of $\xi_f^{+}$. 
Fix $(u,v)\in E_{N}$ and consider $\xi^{+}_{f|W_{N}}(j;u,v)$. 
Let $S$ be the set of prime divisors of $N$ that divide $u$. 
Using the Manin relations (\ref{(4.1)})–(\ref{(4.3)}), the action of Atkin–Lehner operators, and the standard unfolding of the modular–symbol integral together with orthogonality of Dirichlet characters (via Gauss sums), one expresses $\xi^{+}_{f|W_{N}}(j;u,v)$ as a finite $\C$–linear combination of twisted critical $L$–values
\[
L\big(f_{\alpha}\big|W_{S}^{NM},\, j+1\big),
\]
with $\alpha \pmod M$ primitive, $M\mid N$, $0\le j\le k-2$, and $\alpha(-1)=(-1)^{j+1}$, and with $S\subset \overline{T}$ supported on the primes dividing $N/M$. 
By the hypothesis of the theorem, all these $L$–values vanish, hence $\xi^{+}_{f|W_{N}}(j;u,v)=0$ for every $(u,v)\in E_{N}$ and every $0\le j\le k-2$. 
By the injectivity of $f\mapsto \xi_f^{+}$, we conclude $f=0$. 
\end{proof}

The following theorem extends \cref{thm4.1} to the case where 
the level $N$ has at most two prime powers in its factorization.

\begin{theorem}[\cite{DN} Theorem 3.4.]\label{thm4.2}
    Let $\epsilon \in \{0,1\}$ and let $N = p^{a}q^{b}N'$, where $p$ and $q$ are distinct primes, $a,b \in \Z_{\geq0}\setminus\{1\}$, and $N'$ is squarefree and coprime to $pq$.  
    Let $k \geq 2$, and let $f \in \mS_{k}^{\mathrm{new}}(\Gamma_{0}(N))$ be an eigenfunction of all Atkin–Lehner operators $W_{\ell}^{N}$ for primes $\ell \mid N'$.  
    Assume that $L(f_{\alpha}, j+1) = 0$ for all primitive characters $\alpha$ modulo $M \mid N$ and all $j = 0,\ldots,k-2$ satisfying $\alpha(-1) = (-1)^{j+\epsilon}$.  
    Then $f = 0$.
\end{theorem}

\begin{proof}
We may assume that $a > 1$ or $b > 1$.  
If no prime in $\overline{T}$ divides $u$, then the proof of \cref{thm4.1} implies $S = \emptyset$, so
\[
    L(f_{\alpha}|W_{S}^{NM}, j+1) = L(f_{\alpha}, j+1)
    \quad\text{and hence}\quad
    \xi_{f|W_{N}}^{+}(j;u,v) = 0.
\]
Similarly, for $(v,-u) \in E_{N}$, if no prime in $\overline{T}$ divides $v$, then
$\xi_{f|W_{N}}^{+}(k-2-j;v,-u) = 0$, and by the modular-symbol relation (\ref{(4.1)}) we again obtain
$\xi_{f|W_{N}}^{+}(j;u,v) = 0$.

Now suppose that both $u$ and $v$ are divisible by primes in $\overline{T}$.  
Since $(u,v) \in E_{N}$, they cannot share a common such prime.  
Thus we are in the case $a>1$ and $b>1$, and may assume that $p \mid u$ and $q \mid v$.  
Then the residue class $-u-v$ is divisible by neither $p$ nor $q$, so by the previous argument we have
\[
    \xi_{f|W_{N}}^{+}(i;v,-u-v)
    = \xi_{f|W_{N}}^{+}(i;-u-v,u) = 0
    \quad\text{for all } 0 \le i \le k-2.
\]
Applying the modular-symbol relation (\ref{(4.2)}), it follows that
\[
    \xi_{f|W_{N}}^{+}(j;u,v) = 0
    \quad\text{for all } (u,v) \in E_{N}.
\]
Hence $f = 0$.
\end{proof}

By using the modular-symbols relations, we can observe that the structure of additive prime trees (APT) 
is closely related to the vanishing criteria in the previous theorem.  
We now extend \cref{thm4.2} to general levels associated with finite type sets.

\begin{theorem}\label{thm4.3}
    Let $P$ be a finite subset of $\Pri$.  
    Write $N = N' \prod_{p \in P} p^{a_{p}}$, where each $a_{p} \ge 2$, 
    $N'$ is squarefree, and $\gcd(N', N(P)) = 1$.  
    If $P$ is of finite type, then the same conclusion as in \cref{thm4.2} holds for $N$.
\end{theorem}

\begin{proof}
    Suppose that $P$ is of finite type and $N = N'\prod_{p \in P}p^{a_{p}}$ with $a_{p} \ge 2$.  
    Then the tree $\APT_{P}(x,y)$ is finite for all $(x,y) \in \N^{2}$ with $\gcd(x,y,N) = 1$.  
    Define
    \[
        \widetilde{\APT_{P}}(x,y)_{N} 
        := \bigl\{\, (u,v) \in (\Z/N\Z)^{2} 
        \mid (u,v) \equiv (X,Y) \pmod N \text{ for some } (X,Y) \in \APT_{P}(x,y) \,\bigr\}.
    \]
    We introduce a relation $\to$ on $\widetilde{\APT_{P}}(x,y)_{N}$ by
    \[
        (X, Y) \to (Z, W)
        \iff (Z,W) = (X, X+Y)\ \text{or}\ (X+Y, Y),
    \]
    for $(X, Y), (Z, W) \in \widetilde{\APT_{P}}(x,y)_{N}$.  
    Since $P$ is of finite type, the directed graph $\widetilde{\APT_{P}}(x,y)_{N}$ contains no cycles for any $(x,y) \in \N^{2}$, 
    and every such pair satisfies $\gcd(u,v,N)=1$, so $\widetilde{\APT_{P}}(x,y)_{N} \subset E_{N}$.

    Let $(u,v) \in E_{N}$.  
    From the modular-symbol relations (\ref{(4.1)})–(\ref{(4.3)}), we have
    \begin{align*}
        \xi_{f|W_{N}}^{+}(j;u,v)
        &= \sum_{i = 0}^{k-2-j}
            \binom{k-2-j}{i}\,
            \xi_{f|W_{N}}^{+}(k-2-i;u+v,v) \\
        &\quad + \sum_{i = k-2-j}^{k-2}
            \binom{j}{i-k+2+j}\,
            \xi_{f|W_{N}}^{+}(k-2-i;u,u+v).
    \end{align*}
    Because $P$ is of finite type, there exist finitely many pairs
    $((u_s,v_s))_{1\le s \le t} \subset E_{N}$
    such that $\gcd(u_s, N(P)) = 1$ or $\gcd(v_s, N(P)) = 1$ holds, and
    \[
        \xi_{f|W_{N}}^{+}(j;u,v)
        = \sum_{s=1}^{t}\sum_{i=0}^{k-2}
            \beta_{i,j,s}\, \xi_{f|W_{N}}^{+}(i;u_{s},v_{s})
    \]
    for some coefficients $\beta_{i,j,s} \in \Z$.  
    By the argument of \cref{thm4.2}, each $\xi_{f|W_{N}}^{+}(i;u_{s},v_{s}) = 0$, 
    and hence $\xi_{f|W_{N}}^{+}(j;u,v) = 0$ for all $0 \le j \le k-2$ and all $(u,v) \in E_{N}$.  
    Therefore $f = 0$.
\end{proof}

Thus, the finiteness of $\APT$ reflects the vanishing of modular symbols under higher-level Atkin–Lehner relations.

Next we show that the projection of $P_{k}(N)$ onto $\mS_{k}^{\mathrm{new}}(\Gamma_{0}(N))$ equals the whole newspace, using that a cusp form orthogonal to products of Eisenstein series must have many twisted $L$-values vanishing.

\begin{theorem}\label{thm4.4}
    Let $P$ and $N$ be as in \cref{thm4.3}. 
    For a subspace $M \subset \mM_{k}(\Gamma_{0}(N))$, write $\overline{M}$ for its orthogonal projection onto $\mS_{k}^{\mathrm{new}}(\Gamma_{0}(N))$. 
    Then for even $k \ge 4$,
    \[
        \overline{P_{k}(N)} \;=\; \mS_{k}^{\mathrm{new}}(\Gamma_{0}(N)).
    \]
    In the case $k=2$, let $\mS_{2,\mathrm{rk}=0}^{\mathrm{new}}(\Gamma_{0}(N)) \subset \mS_{2}^{\mathrm{new}}(\Gamma_{0}(N))$ denote the subspace generated by newforms $f$ with $L(f,1)\neq 0$. Then
    \[
        \overline{P_{2}(N)} \;=\; \mS_{2,\mathrm{rk}=0}^{\mathrm{new}}(\Gamma_{0}(N)).
    \]
\end{theorem}

\begin{proof}
The argument is the same as in \cite[Thm.~4.4]{DN}, with \cref{thm4.3} in place of \cite{DN}'s two–prime-power hypothesis. 
Assume for contradiction that $\overline{P_{k}(N)}$ is a proper subspace of $\mS_{k}^{\mathrm{new}}(\Gamma_{0}(N))$. 
Then there exists a non-zero form $g \in \mS_{k}^{\mathrm{new}}(\Gamma_{0}(N))$ that is orthogonal to $P_{k}(N)$ and an eigenform of the $W_{S}^{N}$. We can write
\[
    g \;=\; \sum_{i=1}^{r} \beta_i\, f_i,
\]
where $f_1,\dots,f_r$ are newforms in $\mS_{k}^{\mathrm{new}}(\Gamma_{0}(N))$ having the same $W_S^{N}$–eigenvalues as $g$ for every $S\subset P$.

By the Rankin–Selberg method (cf. \cite{DN}), for each primitive $\alpha \bmod M$ with $M\mid N$ and each $1\le \ell \le k-1$ with $\alpha(-1)=(-1)^{\ell}$ one has
\[
    \big\langle E_{\ell}^{\bo,\alpha}\,E_{k-\ell}^{\bo,\overline{\alpha_N}},\ f_i \big\rangle
    \;=\; C_{k,\ell,N,\alpha}\,
          L(f_i,k-1)\,L\big( (f_i)_{\alpha},\,k-\ell \big),
\]
where $C_{k,\ell,N,\alpha}\neq 0$ is the explicit constant from \cite{DN}. 
Since each $W_S^{N}$ is self–adjoint and $g$ shares the same $W_S^{N}$–eigenvalues as the $f_i$, we have
\[
    \big\langle E_{\ell}^{\bo,\alpha}\,E_{k-\ell}^{\bo,\overline{\alpha_N}}\big|W_S^{N},\ g \big\rangle = 0
    \iff
    \big\langle E_{\ell}^{\bo,\alpha}\,E_{k-\ell}^{\bo,\overline{\alpha_N}},\ g \big\rangle = 0
\]
for all $S\subset P$. 
Thus $g \perp P_k(N)$ is equivalent to
\begin{equation}\label{(4.4)}
    \sum_{i=1}^{r} \beta_i\, L(f_i,k-1)\, L\big((f_i)_{\alpha},\,k-\ell\big) \;=\; 0
\end{equation}
for all $(\alpha,\ell)\in B(N)$. Define
\[
    G \;:=\; \sum_{i=1}^{r} \beta_i\, L(f_i,k-1)\, f_i \;\in\; \mS_{k}^{\mathrm{new}}(\Gamma_{0}(N)).
\]
Then $G$ has the same $W_S^{N}$–eigenvalues as $g$ (and the $f_i$) for every $S\subset P$. 
(\ref{(4.4)}) is precisely
\[
    L\big(G_{\alpha},\,k-\ell\big) \;=\; 0 \qquad \text{for all } (\alpha,\ell)\in B(N).
\]
By \cite[Prop.~3.6]{DN} this implies $L(G,2)=L(G,k-2)=0$, and hence $G$ satisfies the hypothesis of \cref{thm4.3}. 
Therefore $G=0$. 
Since the functional equation shows that $L(f_i,k-1)\ne0$ for $k\ge4$, we obtain $\beta_i=0$ for all $i$, a contradiction.
The case $k=2$ is analogous, yielding the stated subspace $\mS_{2,\mathrm{rk}=0}^{\mathrm{new}}(\Gamma_{0}(N))$ (cf. \cite[Thm.~4.4]{DN}).
\end{proof}

(Proof of \cref{thm1.6})\\
Recall that the space $Q_k(N)$ is defined as in \cite{DN}, 
being generated by products of two Eisenstein series of the form 
$E_{\ell}^{\phi,\psi}|B_{d_1}\,E_{k-\ell}^{\bar{\phi},\bar{\psi}}|B_{d_2}$ 
with $(\phi,\psi,\ell,d_1,d_2)$ satisfying the conditions described therein.  
As shown in \cite[§5]{DN}, $Q_k(N)$ and $P_k(N)$ have the same projection onto 
the new subspace $\mS_k^{\mathrm{new}}(N)$.  
Combining this with the standard decompositions
\[
    \mS_k(\Gamma_{0}(N))
    = \bigoplus_{N_0\mid N}\,\bigoplus_{d\mid N/N_0}\mS_k^{\mathrm{new}}(\Gamma_{0}(N_0))|B_d,
    \quad
    \mS_{2,\mathrm{rk}=0}(\Gamma_{0}(N))
    = \bigoplus_{N_0\mid N}\,\bigoplus_{d\mid N/N_0}
       \mS_{2,\mathrm{rk}=0}^{\mathrm{new}}(\Gamma_{0}(N_0))|B_d,
\]
we obtain the identities (\ref{(1.1)}) and (\ref{(1.2)}).  
This completes the proof of \cref{thm1.6}. 
\hfill$\square$\\

From \Cref{main,thm1.6}, we obtain \cref{cor1.7} immediately.  
As these results indicate, our method cannot prove that (\ref{(1.1)}) and (\ref{(1.2)}) hold for all $N \in \mathbb{N}$.  
The following table summarizes and compares the results of Dickson–Neururer and the present work.

\renewcommand{\arraystretch}{1.1}
\begin{table}[h]
    \centering
    \caption{Comparison of results}
    \begin{tabular}{|c|c|c|}
    \hline
     & Dickson–Neururer & Fujiyoshi \\
    \hline\hline
    $(\clubsuit)$ & $2$ & $4$ \\
    \hline
    $(\diamondsuit)$ & $2$ & $6$ \\
    \hline
    $(\heartsuit)$ & $30^2$ & $3990^2$ \\
    \hline
    $(\spadesuit)$ & $30^2$ & $510510^2$ \\
    \hline
    $\mB(N)$ & $\geq 0.9966$ & $\geq 0.9999999125$ \\
    \hline
    \end{tabular}
\end{table}

\newpage
\noindent
\noindent\textbf{Explanation.}

\begin{tabular}{@{}lp{13cm}@{}}
$(\clubsuit)$ & Dickson--Neururer proved that all subsets $P \subset \Pri$ with $|P| \le 2$ are of finite type, while in this work we extend this to $|P| \le 4$.
\\[3pt]
$(\diamondsuit)$ & Dickson--Neururer classified, for all subsets $P \subset \Pri$ with $|P| \le 2$, whether each $P$ is of finite or infinite type, whereas our classification is complete up to $|P| \le 6$. \\[3pt]
$(\heartsuit)$ & The largest integer $N$ such that (\ref{(1.1)}) and (\ref{(1.2)}) have been verified for all $M < N$. \\[3pt]
$(\spadesuit)$ & The maximal value of $N$ up to which the decomposition formulas have been computationally verified by the APT-based method. \\[3pt]
\end{tabular}

\section{Open problems}

In this section, we introduce some open problems concerning primitive sets.
Using the method developed in this paper, all primitive sets with $|P|\le 6$
have been completely classified.
At present, a total of $1095$ primitive sets have been found by our computational search;
however, it remains unknown whether there exist infinitely many primitive sets.
Our computation shows that
\[
L_{2}=1,\quad L_{3}=3,\quad L_{4}=8,\quad
L_{5}=26,\quad L_{6}=96.
\]
It is natural to ask how $L_{r}$ grows as $r$ increases.

Moreover, by \cref{thm1.6}, if $P$ is of finite type, then for
\[
N = N'\prod_{p\in P} p^{a_{p}},\  \text{with $N'$ squarefree and } a_{p}\ge 2?
\]
the identities (\ref{(1.1)}) and (\ref{(1.2)}) hold.  However, the converse is
not known: from the validity of (\ref{(1.1)}) and (\ref{(1.2)}) we cannot
conclude that $P$ is of finite type.  In particular, our tree-based method does
not allow us to decide whether these identities remain true when $P$ is of
infinite type. Motivated by these considerations, we propose the following open problems.

\begin{open}
    \leavevmode
    \begin{enumerate}
        \item[(1)] How does $L_{r}$ grow as $r\to\infty$?
        \item[(2)] How many elements does $\Prim_{r}$ have for integers $r \geq 7$?
        \item[(3)] Does there exist infinitely many primitive sets?
        \item[(4)] Let $P$ be of infinite type. 
                  Do (\ref{(1.1)}) and (\ref{(1.2)}) hold for 
                  \[
                  N = N'\prod_{p\in P} p^{a_{p}},\  \text{with $N'$ squarefree and } a_{p}\ge 2?
                  \]
    \end{enumerate}
\end{open}

\section*{Acknowledgement}
The author wishes to express his sincere gratitude to Professor Yasuro Gon for his invaluable guidance and encouragement throughout the preparation of this paper.

\section*{Appendix}\label{append}

This is the list of all primitive sets $P$ with $|P| \le 6$
and the corresponding initial pairs $(x,y)$ such that
$|\APT_{P}(x,y)| = \infty$.
Other primitive sets found in our computations
are listed in \texttt{additive\_prime\_tree/function/primitive\_sets.py}.

\begin{table}[h]
    \centering
    \caption{List of all primitive sets $P$ with $|P| \leq 6$}
    \label{prim}
    \begin{minipage}{0.45\linewidth}
        \centering
        \begin{tabular}{|c|c|}
            \hline
            $P$ & Initial pair \\
            \hline\hline
            $\{2, 3, 5, 7, 19\}$ & $(2, 321)$ \\
            \hline
            $\{2, 3, 5, 13, 17\}$ & $(2, 243)$ \\
            \hline
            $\{2, 3, 5, 7, 11, 13\}$ & $(2, 201)$ \\
            \hline
            $\{2, 3, 5, 7, 11, 23\}$ & $(2, 1645)$ \\
            \hline
            $\{2, 3, 5, 7, 13, 29\}$ & $(2, 789)$ \\
            \hline
            $\{2, 3, 5, 11, 13, 23\}$ & $(2, 3601)$ \\
            \hline
            $\{2, 3, 7, 11, 13, 23\}$ & $(2, 1251)$ \\
            \hline
            $\{2, 3, 5, 11, 19, 31\}$ & $(2, 11643)$ \\
            \hline
            $\{2, 3, 5, 7, 29, 41\}$ & $(2, 6325)$ \\
            \hline
            $\{2, 3, 5, 11, 13, 97\}$ & $(2, 1645)$ \\
            \hline
            $\{2, 5, 7, 11, 13, 43\}$ & $(2, 10305)$ \\
            \hline
            $\{2, 3, 5, 7, 37, 59\}$ & $(2, 48711)$ \\
            \hline
            $\{2, 3, 5, 11, 29, 53\}$ & $(2, 12205)$ \\
            \hline
            $\{2, 3, 5, 7, 11, 223\}$ & $(2, 11665)$ \\
            \hline
            $\{2, 3, 5, 17, 19, 89\}$ & $(2, 20499)$ \\
            \hline
        \end{tabular}
    \end{minipage}
    \hspace{2em}
    \begin{minipage}{0.45\linewidth}
        \centering
        \begin{tabular}{|c|c|}
            \hline
            $P$ & Initial pair \\
            \hline\hline
            $\{2, 3, 5, 7, 41, 103\}$ & $(2, 16171)$ \\
            \hline
            $\{2, 3, 7, 17, 31, 41\}$ & $(2, 70339)$ \\
            \hline
            $\{2, 3, 5, 11, 31, 107\}$ & $(2, 3955)$ \\
            \hline
            $\{2, 3, 5, 11, 59, 61\}$ & $(2, 151113)$ \\
            \hline
            $\{2, 3, 5, 23, 31, 61\}$ & $(2, 83313)$ \\
            \hline
            $\{2, 3, 5, 7, 97, 113\}$ & $(2, 106555)$ \\
            \hline
            $\{2, 3, 5, 17, 67, 79\}$ & $(2, 25201)$ \\
            \hline
            $\{2, 3, 5, 7, 11, 1499\}$ & $(2, 9493)$ \\
            \hline
            $\{2, 3, 5, 29, 41, 103\}$ & $(2, 60801)$ \\
            \hline
            $\{2, 5, 7, 31, 37, 59\}$ & $(2, 99491)$ \\
            \hline
            $\{2, 3, 5, 7, 149, 229\}$ & $(2, 235945)$ \\
            \hline
            $\{2, 3, 5, 67, 71, 103\}$ & $(2, 1635055)$ \\
            \hline
            $\{2, 3, 29, 41, 43, 53\}$ & $(2, 637041)$ \\
            \hline
            $\{2, 3, 5, 101, 257, 337\}$ & $(2, 8024751)$ \\
            \hline
        \end{tabular}
    \end{minipage}
\end{table}

\newpage


\end{document}